\documentclass[11pt]{amsart}
\usepackage{latexsym}
\usepackage{graphics}
\usepackage{graphicx}
\usepackage{epstopdf}
\usepackage{tikz}
\usepackage{amsmath}
\usepackage{ulem}
\usepackage{float}
\usepackage{color}
\usepackage{caption}
\usepackage[labelformat=simple]{subcaption}

\newtheorem{theorem}{Theorem}[section]
\newtheorem{proposition}[theorem]{Proposition}
\newtheorem{lemma}[theorem]{Lemma}
\newtheorem{corollary}[theorem]{Corollary}

\theoremstyle{remark}
\newtheorem{remark}[theorem]{Remark}
\newtheorem{example}[theorem]{Example}

\title{Tangle decompositions of alternating link complements}
\author{Joel Hass, Abigail Thompson, Anastasiia Tsvietkova}
\date{}
\subjclass[2010]{}
\everymath{\displaystyle}
\begin{document}

 \footnotesize
 \begin{abstract} {Decomposing knots and links into tangles is a useful technique for understanding their properties.   The notion of prime tangles was introduced by Kirby and Lickorish in \cite{KLic};  Lickorish proved \cite{Lickorish} that by summing prime tangles one obtains a prime link.  In a similar spirit, summing two prime alternating tangles will produce a prime alternating link, if summed correctly with respect to the alternating property. Given a prime alternating link, we seek to understand whether it can be decomposed into two prime tangles each of which is alternating.     We refine results of Menasco and Thistlethwaite to show that if such a decomposition exists either it is visible in an alternating link diagram or the link is of a particular form, which we call a pseudo-Montesinos link.}
 \end{abstract}

\maketitle
\normalsize

\section{Overview}

We review some definitions and give an outline of the paper.

Let $L$ be a non-split prime alternating link in $S^3$.  A properly imbedded surface in the complement of $L$ is \textit{essential}  if it is incompressible, boundary incompressible and non-boundary parallel  in $S^3-L$.

{A \textit{Conway sphere} for $L$ is an essential 4-punctured sphere properly imbedded in the complement of $L$ with meridianal boundary components.  We say that a Conway sphere splits $L$ into two \textit{2-tangles}, i.e., two 3-balls each of which contains two strands of $L$. A 3-ball may also additionally contain components of $L$ disjoint from $F$. Since $L$ is prime by hypotheses, the tangles on each side are \textit{prime}.  A 2-tangle in which the two strands are boundary parallel is called a  \textit{rational} tangle.

The notion of a prime tangle was suggested by Kirby and Lickorish in \cite{KLic}. Lickorish
shows in \cite{Lickorish} that by summing prime tangles one obtains a prime link. Similarly, summing
two prime alternating tangles produces a prime alternating link, if summed correctly with
respect to the alternating property. Here, we look at the converse: given a prime alternating
link diagram, can one determine whether it is a sum of two prime alternating tangles?

Menasco proved that a Conway sphere is realized in a reduced alternating diagram as either visible (represented by a PPPP curve in standard position) or hidden (two PSPS curves) in \cite{Menasco1984}. He
also proved that two PSPS curves in a prime alternating diagram represent a 4-punctured
sphere that is essential \cite{Menasco1985}. It is known how to determine whether a given
PPPP curve (visible case) represents a 4-punctured sphere that is essential: see Menasco and Thislethwaite \cite{MT93}, and Thislethwaite \cite{Thistlethwaite}. All these facts together give a
purely diagrammatic algorithm to determine whether the link has a decomposition into two
prime tangles, though possibly not alternating.

We revisit the case of a ``hidden" Conway sphere.  We prove that such a sphere forces the presence of another visible Conway sphere right next to it, in the same diagram, unless the link is a pseudo-Montesinos link (Proposition 3.2). Pseudo-Montesinos links will be defined in Section 3 and are a subset of arborescent links \cite{BonSieb}. This was not known before: rather, Thistlethwaite noted that a visible Conway sphere is always visible in an alternating diagram (\textit{i.e.} if visible in one alternating diagram, then also visible in other alternating diagrams of the same link), and therefore a hidden one is always hidden \cite{Thistlethwaite}. We also prove that a pseudo-Montesinos link has no visible Conway spheres in any alternating diagram (Proposition 3.3 with corollaries). Our two results together yield that a decomposition into two prime alternating tangles can always be detected by looking at an alternating diagram (Theorem 3.1). In the final section we also show that for a closed braid $L$, detecting prime alternating tangle decompositions takes an even easier form. Indeed, once the decomposing sphere is essential, it must be positioned in a special way with respect to $D$ (Proposition 4.1).

Our proofs mainly use two techniques. The first one is inherited from Menasco's work: we use standard position, which helps to translate topology of a surface into combinatorics of curves and diagrams. The second component is a careful topological analysis of certain isotopies and strong isotopies of surfaces embedded in 3-manifolds.

  The visibility of a prime alternating tangle decomposition (and of the genus-2 surfaces that yield an  essential 4-punctured sphere after meridianal compressions) aligns with well-known results of Menasco, who noted that some of the basic topological properties of alternating link complements can be seen directly in reduced alternating link diagrams \cite{Menasco1984}. Among them is the property of a link being non-split (and respectively the presence of an essential genus-0 surface in the link complement), and the property of being prime (and the presence of an essential genus-1 surface).

In Section \ref{IntersectionPattern}, we recall results of Menasco and Thistlethwaite on surfaces in alternating link complements.   In Section \ref{Decompositions} we state and prove our main theorem.   In the final Section \ref{Braids} we closely examine tangle decompositions of alternating braids.

  \section{Conway spheres and standard position}\label{IntersectionPattern}

  In this preliminary section, we recall Menasco's techniques and a key lemma, as well some of the results of Thistlethwaite on rational tangles.

Let $D$ be a reduced alternating diagram of the link $L$.   Let  $Q$ be the projection sphere where $D$ lies except for perturbations at crossings. We review the notion of a surface in  standard position  (see Section 2 of \cite{Menasco1984}).

The link $L$ lies on a union of two spheres in $S^3$, $S^2_+$ and $S^2_-$, which agree with $Q$ except in a bubble around each crossing. At the bubbles, $S^2_+$ and $S^2_-$ go over the top and bottom hemispheres respectively. We will denote by $B_+$ and $B_-$ the parts of $S^3$ lying above and below $Q$ respectively.

Let $F$ be an essential surface properly embedded in the complement of $L$ such that every boundary component of $F$ is meridianal. If $F$ is closed, we meridianally compress it until no further meridianal compression is possible. Then $F\cap S^2_\pm$ consists of simple closed curves bounding disks of $B_{\pm}\cap F$. Following Menasco's technique, we encode such a curve of intersection $C$ by a word consisting of the letters $P$ and $S$. {The letter $P$ means $C$ intersects a strand of a link not at a crossing, \textit{i.e.} $F$ has a meridianal boundary component there. The letter $S$ means $C$ intersects a crossing, \textit{i.e.} $F$ passes between two strands of a crossing and is shaped as a saddle there. Fig.\ref{PSPS} depicts an example of a $PSPS$ curve, projected on $Q$ from $F\cap S^2_+$ (a fragment of a link diagram is pictured in grey). There are multiple ways to put a surface $F$ in standard position. If it is done so that the total number of $S$'s, $P$'s and the curves of $F\cap S^2_\pm$ is minimized, then $F$ is in \textit{complexity minimizing standard position}.

\begin{figure}[ht]
\centering
\begin{subfigure}[b]{0.47 \textwidth}
\centering
\includegraphics[scale=0.6]{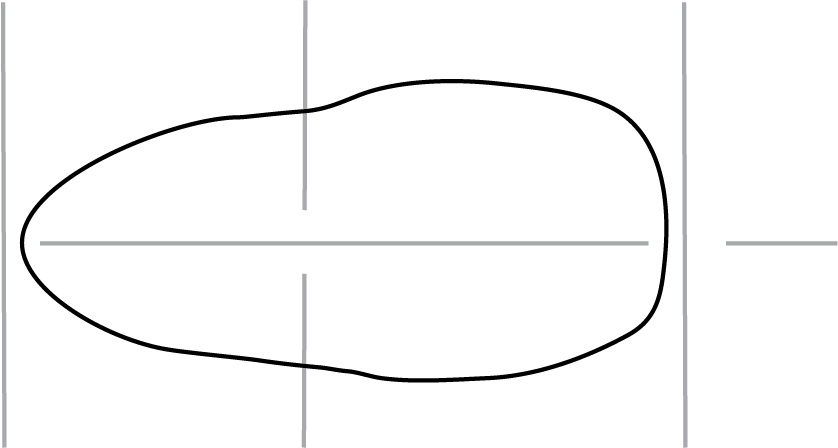}
\caption{A $PSPS$ curve.}\label{PSPS}
\end{subfigure}
\begin{subfigure}[b]{0.47 \textwidth}
\centering
\includegraphics[scale=0.6]{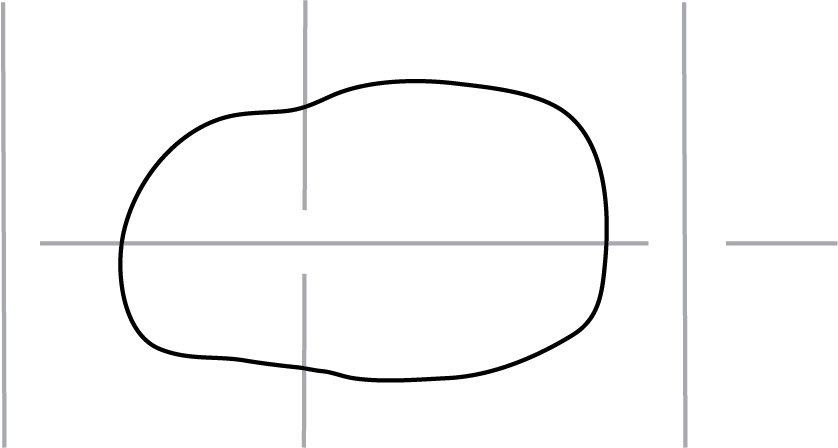}
\caption{A modified $PSPS$ curve.}
\label{PSPSmodified}
\end{subfigure}
\caption{}
\end{figure}

The following observations follow from the techniques of Menasco (\cite{Menasco1984}). For details, see \cite{HTT}.

A segment of $D$ from a crossing to an adjacent crossing will be called \textit{a edge}.

\begin{lemma}\label{Lemma} Suppose $F$ is a closed essential 4-punctured sphere in the complement $S^3-L$ of a prime alternating non-split link $L$, i.e., $F$ is a Conway sphere for $L$.    Then $F$ can be placed in standard position relative to $S^2_+ \cup S^2_-$ so that
\begin{enumerate}
\item\label{GenusTwoWords} $F$ intersects $S^2_+$ in either a single $PPPP$ curve or in  two $PSPS$ curves.
\item\label{SaddleArc} No curve passes through a saddle and then crosses an edge of $D$ adjacent to the saddle.

 \item\label{ArcTwice} No curve crosses an edge of $D$ twice consecutively.

    \end{enumerate}
\end{lemma}

\begin{remark}\label{RationalTangle}

Menasco's techniques found application in the work of Thistlethwaite on rational tangles \cite{Thistlethwaite}, in which he describes precisely what an alternating diagram of a rational tangle looks like, as follows:

Start with a diagram of a 2-string tangle that has no crossings. Then surround this diagram by annuli,
each annulus containing four arcs of the link diagram joining distinct boundary components of the annulus, and connected to the arcs in the neighboring annuli or the described 2-string tangle. Each annulus contains  a single crossing between two of the arcs. See the comments after Corollary 3.2 in \cite{Thistlethwaite}, as well as Fig.2 in \cite{LackenbyTunnel} for an illustration. In particular, one can determine whether an alternating tangle diagram represents a rational tangle just by looking at the diagram.
\end{remark}

\section{Decomposition into two prime alternating tangles}\label{Decompositions}

A \textit{Montesinos link} is a link obtained by taking a cyclic sum of a finite number of rational 2-tangles (see top of Fig.\ref{pseudo} for an example; inside each circle insert a diagram of a rational tangle). We will often refer to these four tangles as \textit{sub-tangles}, since together they may form larger tangles that we consider. Given a Montesinos link with four rational sub-tangles $T_i, i=1,2,3,4$, we construct a \textit{pseudo-Montesinos link} as follows: delete four strands, one each connecting $T_1$ to $T_2$,  $T_2$ to $T_3$, $T_3$ to $T_4$, and $T_4$ to $T_1$.  Then, following the pattern shown in the bottom of Fig.\ref{pseudo}, replace these strands with four strands, two between $T_1$ and $T_3$ and two between $T_2$ and $T_4$. If the resulting diagram is reduced and alternating and each rational sub-tangle has at least one crossing we say that it is a  \textit{standard diagram of an alternating pseudo-Montesinos} link.   Fig.\ref{Link2} is an example of standard diagram of an alternating pseudo-Montesinos link.

\begin{figure}[h]
    \centering
    \includegraphics[width=0.7\textwidth]{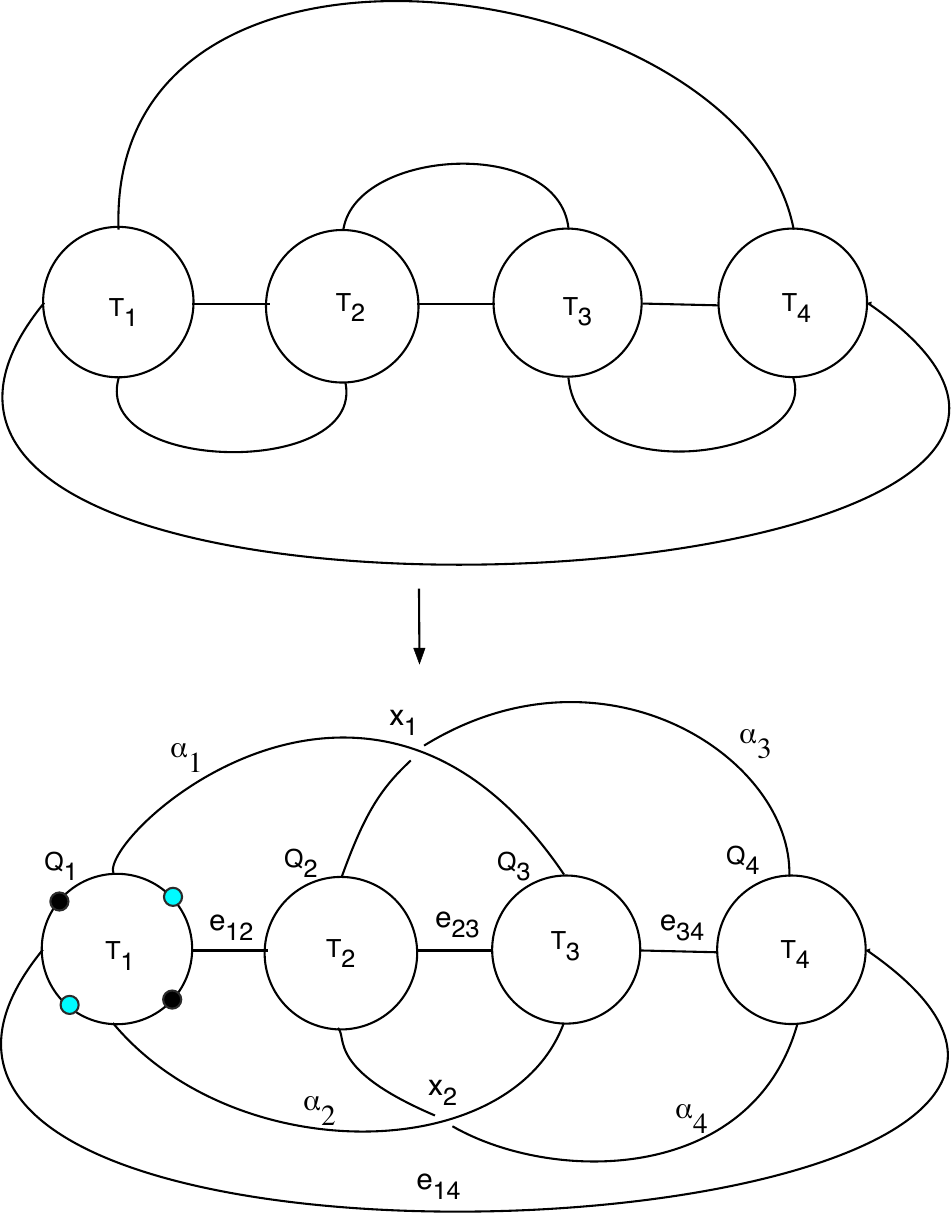}
    \caption{Constructing a pseudo-Montesinos link.}
    \label{pseudo}
\end{figure}

In the terminology of Thistlethwaite (\cite{Thistlethwaite}), a \textit{visible} 4-punctured sphere in an alternating diagram $D$ is one that appears in the plane of the diagram (after isotopy to standard position) represented by a $PPPP$ curve. A \textit{hidden} Conway sphere is one that is represented by two $PSPS$ curves. We extend this to call any 2-tangle $T$ in $L$ \textit{visible} if a $PPPP$ curve $c$ intersects four arcs of $D$ such that all of $T$ lies on one side of $c$, and the complement of $T$ (denote it by $T^c$) lies entirely on the other side of $c$.

Our main theorem is the following:

\begin{theorem}\label{Theorem}
Suppose $L$ is a prime alternating non-split link, $D$ is a reduced alternating diagram for $L$, and there is an essential Conway sphere embedded in $S^3-L$. Then a prime tangle decomposition of $L$ is visible in $D$ if and only if $D$ is not a standard  diagram of an alternating pseudo-Montesinos link.    Further, if $D$ is a standard  diagram of an alternating pseudo-Montesinos link, then no prime tangle decomposition for $L$ is visible in any reduced alternating diagram for $L$.
\end{theorem}

The rest of the section is devoted to the proof of Theorem \ref{Theorem} via a sequence of propositions.

\begin{proposition}\label{PropOne} Suppose $L$ is a prime alternating non-split link, $D$ is a reduced alternating diagram for $L$, and there is an essential Conway sphere embedded in $S^3-L$. Then either a prime tangle decomposition of $L$ is visible in $D$, or $D$ is a standard diagram of an alternating pseudo-Montesinos link.
\end{proposition}

\begin{proof}
Let $Z$ be an embedded 4-punctured sphere splitting $D$ into two prime 2-tangles $R$ and $Q$. Choose one of the two spheres $S^2_-$ and $S^2_+$, say $S^2_+$, and consider its intersections with $Z$. By Lemma \ref{GenusTwoWords}, we can place $Z$ in standard position so that either $Z$ intersects $S^2_+$ in a single $PPPP$ curve or in exactly two $PSPS$ curves. 

If $Z$ intersects $S^2_+$ in a single $PPPP$, then the tangle decomposition is visible in the diagram and we are done.

Assume $Z$ intersects  $S^2_+$ in two $PSPS$ curves. The two curves naturally divide the diagram into four 2-string ``sub-tangles"  (see Fig.\ref{Link1}) along simple closed curves that intersect the link in four points.  One of these sub-tangles is labeled $T$ in the figure, with the simple closed curve $c$ (depicted by the dotted line) on its boundary. Note that all four sub-tangles are visible; by standard position and Lemma \ref{Lemma} (\ref{SaddleArc}),  note that each sub-tangle contains at least one crossing.

\begin{figure}[ht]
\centering
\begin{subfigure}[b]{0.46 \textwidth}
\centering
\includegraphics[scale=0.7]{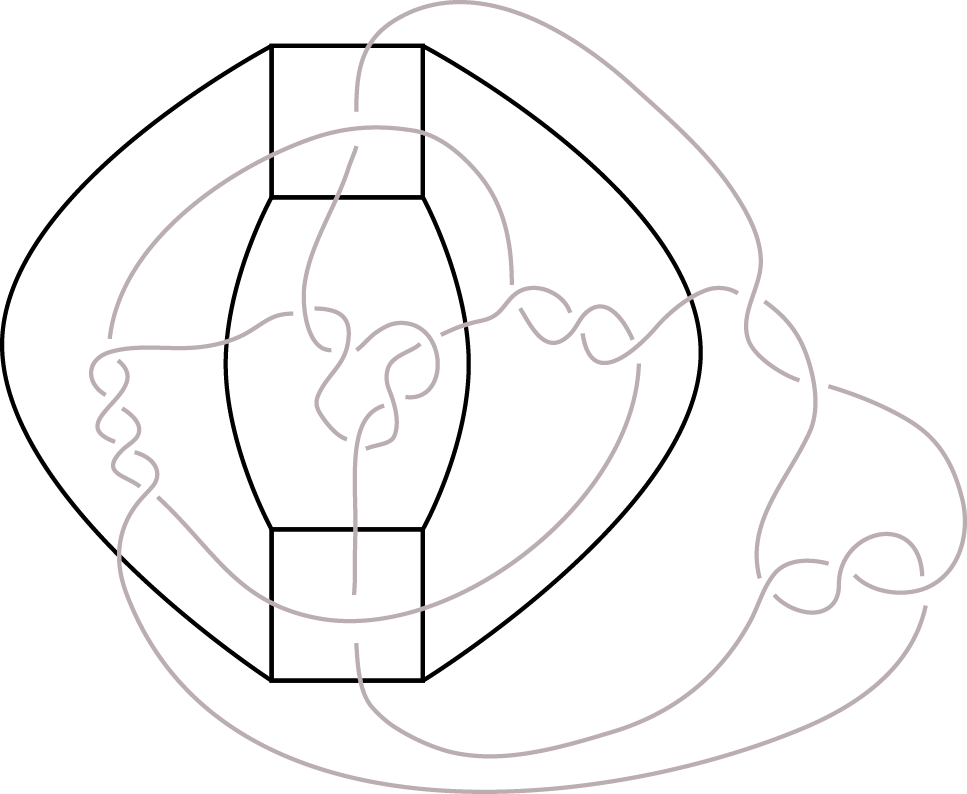}
\caption{}
\label{Link1}
\end{subfigure}
\begin{subfigure}[b]{0.46 \textwidth}
\centering
\includegraphics[scale=0.7]{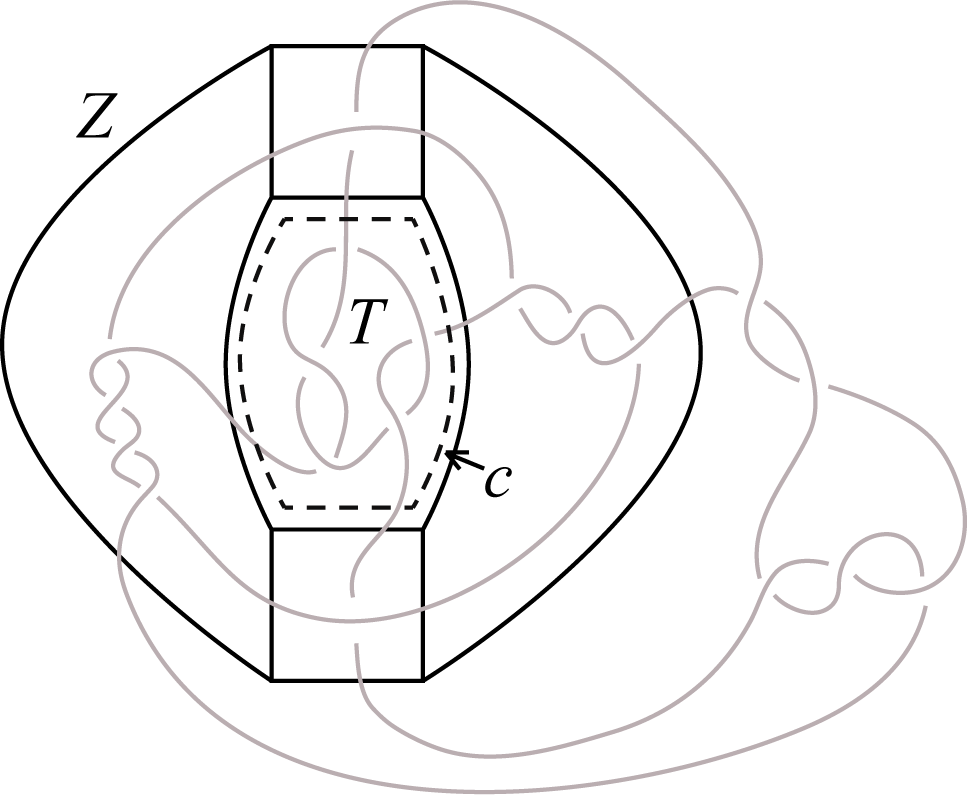}
\caption{}
\label{Link2}
\end{subfigure}
\caption{Two $PSPS$ curves in a link diagram.}
\end{figure}

There are two possibilities:

\
1. Each of these sub-tangles is a rational tangle.

\
2.  At least one of the sub-tangles is prime.

\noindent  We examine each sub-tangle.  By Remark \ref{RationalTangle}, we can determine whether each of the sub-tangles is rational just by looking at the link diagram.  If possibility 1 holds, then $D$ is a standard  diagram of an alternating pseudo-Montesinos link and we are done.   Otherwise at least one of the sub-tangles is prime.

 Assume at least one of the sub-tangles $T$ is prime as in Fig.\ref{Link1}.  We will show that the complementary tangle $T^c$ is also prime.  This proves that the curve $c$ (which is a $PPPP$ curve) that we see on the diagram $D$ describes a decomposition of $L$ into two prime tangles.

Capping the curve $c$ with disks on both sides of the projection sphere forms a 4-punctured 2-sphere $W$.  We call $W$ the 2-sphere {\it associated to $c$}.  $W$ splits the knot into the two 2-tangles, $T$ and $T^c$.  Since $W$ and $Z$ can be assumed disjoint,  one of the two tangles defined by $Z$ therefore also lies completely inside $T^c$.  Assume the tangle $Q$ lies completely inside $T^c$.

We claim that $T^c$ cannot be a rational tangle. Assume to the contrary that  $T^c$ is rational. Then the two arcs in $T^c$  are parallel through disks $E_1$ and $E_2$ to the 4-punctured sphere $W$.    We consider how these disks intersect the 4-punctured sphere $Z$. There are two cases; $Z$ intersects both arcs of the tangle $T^c$ (in two points each), or $Z$ intersects one arc of $T^c$ (in four points) and is disjoint from the other.

In either case,  we use the fact that $Z$ is incompressible in the complement of the knot to remove simple closed curves of intersection between $E_1\cup{E_2}$ and $Z$.   An outermost (in $E_1\cup{E_2}$) arc of intersection with $Z$ can then be doubled to yield a compressing disk for $Z$ in the complement of the link.  Since $Z$ is incompressible in the link complement, this is a contradiction.
Hence $T^c$ is not rational.

A similar argument shows $T^c$ cannot contain any essential twice-punctured sphere, so $T^c$ is a prime tangle, as required. \end{proof}

Note that every standard diagram of an alternating pseudo-Montesinos link gives rise to two $PSPS$ curves, as in Fig.\ref{Link1}. By Theorem 2 of \cite{Menasco1985} the resulting sphere is essential, i.e. every standard diagram of an alternating pseudo-Montesinos link has a hidden Conway sphere. We now consider the possibility of a visible Conway sphere.

Suppose $D$ is a standard diagram of an alternating pseudo-Montesinos link $L$. Let $c$ be a $PPPP$ curve in $D$, and let $W$ be the 4-punctured 2-sphere associated to $c$.     A {\it strong isotopy} of $W$rel$D$ is an isotopy of $W$ that induces a planar isotopy of $c$.     We say that $c$ is a {\it flyping curve} for $D$ if there exists another $PPPP$ curve $b$ disjoint from $c$ and the annulus between $c$ and $b$ contains a single crossing of the diagram. Then there is a flype of the diagram, which changes which strands have the single crossing between $c$ and $b$, and which turns the tangle inside $c$ upside down. We call this {\it flyping the diagram along $c$};  this preserves the alternating property (see \cite{MT}, Fig.1 for more details; our ``flyping curve'' is the boundary of the tangle $S_A$). Notice that if the disk bounded by $c$ which is also contained in $b$ only contains 0 or 1 crossing, flyping along $c$ leaves the diagram unchanged. The following results use the labeling from Fig.\ref{pseudo}; in particular, $Q_i$ is the closed punctured curve bounding the tangle $T_i$ in $D$.

\begin{proposition}\label{PropTwo} Suppose $D$ is a standard diagram of an alternating pseudo-Montesinos link $L$. Let $c$ be a $PPPP$ curve in $D$ and let $W$ be the 4-punctured 2-sphere associated to $c$.    Then there is a strong isotopy of $W$rel$D$ such that $c$ is isotoped to a $PPPP$ curve $c'$ which is either parallel to  $Q_i$ for some $i$ or which lies completely inside  $T_i$ for some $i$, or $c’$ bounds a disk which is disjoint from all the $T_i$’s and which contains zero or one crossing.    In the last case note that flyping along $c'$ does not change the diagram.\end{proposition}

Assuming this proposition, we can prove:

{\begin{corollary}\label{CorOne} Suppose $D$ is a standard diagram of an alternating pseudo-Montesinos link $L$. Let $c$ be a $PPPP$ curve in $D$ and let $W$ be the 4-punctured 2-sphere associated to $c$.   Then $W$ is not a Conway sphere for $L$.  If $c$ is a flyping curve for $D$,  flyping along $c$ yields another standard alternating pseudo-Montesinos diagram for $L$.   \end{corollary}

\begin{proof}
By Proposition \ref{PropTwo}, there is strong isotopy of $W$rel$D$ such that $c$ is isotoped to a $PPPP$ curve $c'$ which is parallel to  $Q_i$ for some $i$, or which lies completely inside  $T_i$ for some $i$, or bounds a disk containing a single crossing  completely exterior to $\cup{T_i}$.    In all cases, $c$ bounds a rational tangle.  If $c$ is a flyping curve for $D$,  flyping along $c$ preserves the rationality of $T_i$, hence preserves the standard alternating pseudo-Montesinos diagram structure of $D$. \end{proof}

\begin{corollary}\label{CorTwo} Suppose $D$ is a standard diagram of an alternating pseudo-Montesinos link $L$.   Then no reduced alternating diagram for $L$ contains a $PPPP$ curve $c$ corresponding to a Conway sphere.      \end{corollary}

\begin{proof}
Since any two reduced alternating diagrams are related by a sequence of flypes \cite{MT}, Corollary \ref{CorTwo} follows from Corollary \ref{CorOne}. \end{proof}

Theorem \ref{Theorem} follows immediately from Propositions \ref{PropOne} and Corollary \ref{CorTwo}.

The remainder of this section is devoted to the proof of Proposition \ref{PropTwo}.

\begin{proof}
Minimize the number of points of intersection $c\cap{\bigcup{Q_i}}$ up to strong isotopy of  $W$rel$D$.

Case 1: $c\cap{\bigcup{Q_i}}=\emptyset$.

\begin{itemize}

{\item Subcase a: $c\subset{T_i}$ for some $i$; then we are done.}\\
{\item Subcase b: $c\cap{\bigcup{T_i}}=\emptyset$.}

\end{itemize}
Proof for Subcase b:\\
Up to relabeling, one of the following must hold:\\

\noindent i. $c$ does not separate the $T_i$'s.\\
ii.  $c$ separates $T_1$ from $T_2, T_3,T_4$.\\
iii. $c$ separates the $T_i$'s in pairs.\\

i.   Then $c$ either contains a single crossing or it bounds a 2-stranded tangle with no crossings. In both cases, flyping over $c$ leaves $D$ unchanged. 

ii. Then $c$ is parallel to $Q_1$ and we are done.

iii.  In both cases ($c$ separates $T_1$ and $T_2$ from the rest, or $T_1$ and $T_3$ from the rest), examination of Fig.\ref{pseudo} shows $c$ must intersect $D$ in at least 6 points, a contradiction.

Case 2: $c\cap{\bigcup{Q_i}}\neq\emptyset$.    Assume $c$ intersects $Q_1$.

The points of intersection between $c$ and $Q_1$ divide $c$ into subarcs $\beta_1, \beta_2,....,\beta_{2m}$ with $\beta_{2j}\subset{T_1}$.

Since we have minimized the number of points of intersection in $c\cap{\bigcup{Q_i}}$ up to strong isotopy of $W$, $\beta_i\cap{D}\neq{\emptyset}$ for each $i$. Therefore there are at most four such subarcs of $c$.

If $\beta_j\subset{T_1}$ only intersects $D$ in a single point for some $j$, we can (strongly) isotop $\beta_j$ out of $T_1$, contradicting minimality. Therefore there are at most two such subarcs of $c$, $\beta_1$ and $\beta_2$, with $\beta_2\subset{T_1}$, and each subarc intersects $D$ in exactly two points.

The endpoints of $\beta_2$  must separate the intersection points of $D$ with $Q_1$ into two pairs.
 We now apply the above arguments to $\beta_1$, which shares these two endpoints with $\beta_2$.  It is useful to refer back to Fig.\ref{pseudo}, where the endpoints of  $\beta_2$ (and $\beta_1$) are marked on $Q_1$, as either the pair of black points or the pair of blue points.    We note that $\beta_1$ cannot intersect any other $T_i$, since any subarc of $c$ contained in a $T_i $ must intersect $D$ in at least two points, and the sections of $\beta_1$ disjoint from all $T_i$'s would also have to intersect $D$ at least once, a contradiction.  Hence $\beta_1$ must be a subarc disjoint from $T_2\cup T_3\cup T_4$, with both endpoints on  $Q_1$, intersecting $L$ in exactly two points.   By inspection we see that $\beta_1$ can be strongly isotoped into $T_1$, contradicting minimality.    Hence $c\cap{\bigcup{Q_i}}=\emptyset$. \end{proof}

 \section{Tangle decompositions of alternating braids}\label{Braids}

Let $L$ be an alternating, prime, non-split closed $n$-braid with a reduced alternating braid diagram $D$, $n\geq 3$.
For a diagram $D$, see Fig.\ref{BraidDiagram}, where every square represents either a twist or a 2-tangle with no crossings. By a \textit{twist} we mean either a single crossing, or a connected sequence of bigon regions of $D$ that is not a part of another such sequence. The point $x$ in the figure represents the braid axis. 

 \begin{figure}[ht]
 \centering
 \begin{subfigure}[b]{0.4 \textwidth}
 \centering
 \includegraphics[scale=0.48]{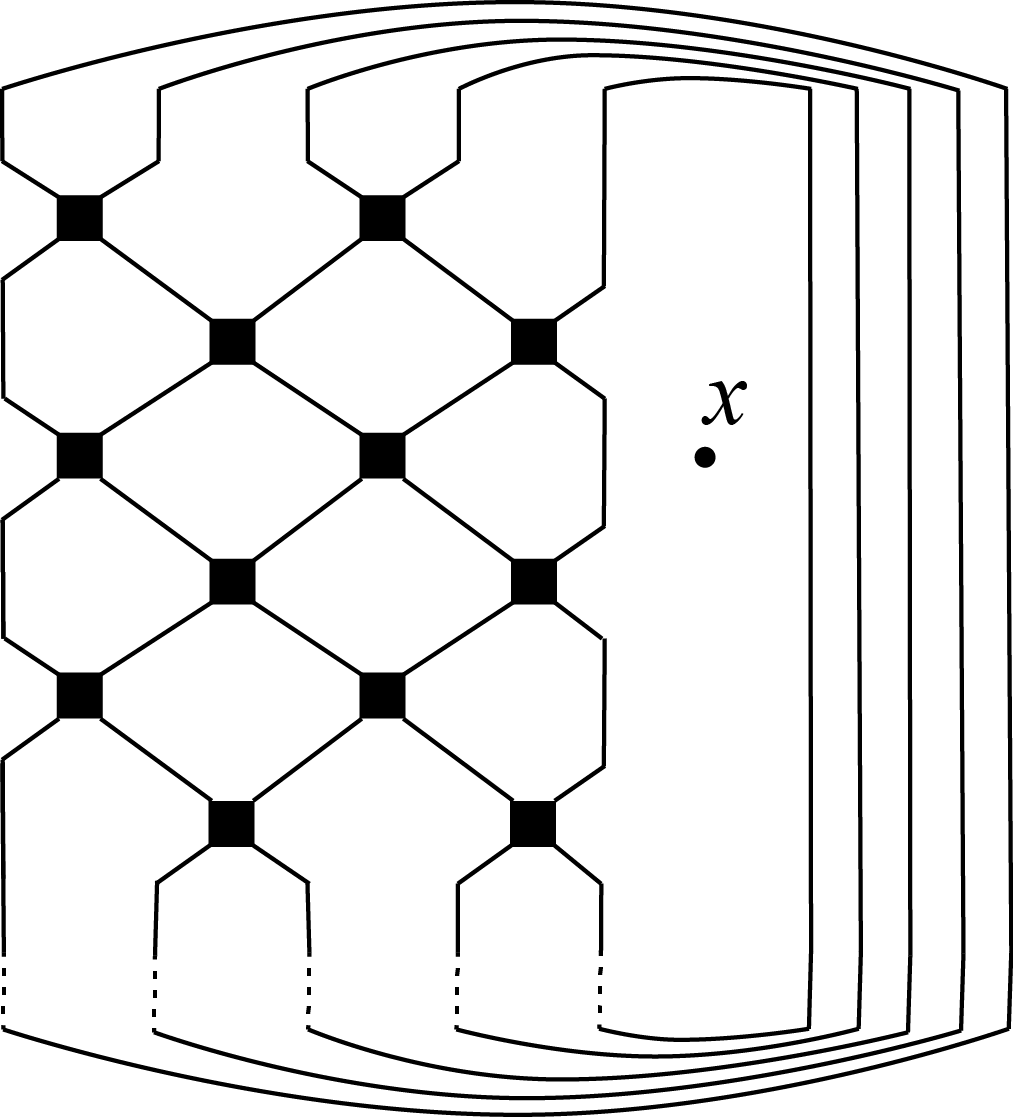}
 \caption{}
 \label{BraidDiagram}
 \end{subfigure}
 \begin{subfigure}[b]{0.4 \textwidth}
\centering
 \includegraphics[scale=0.49]{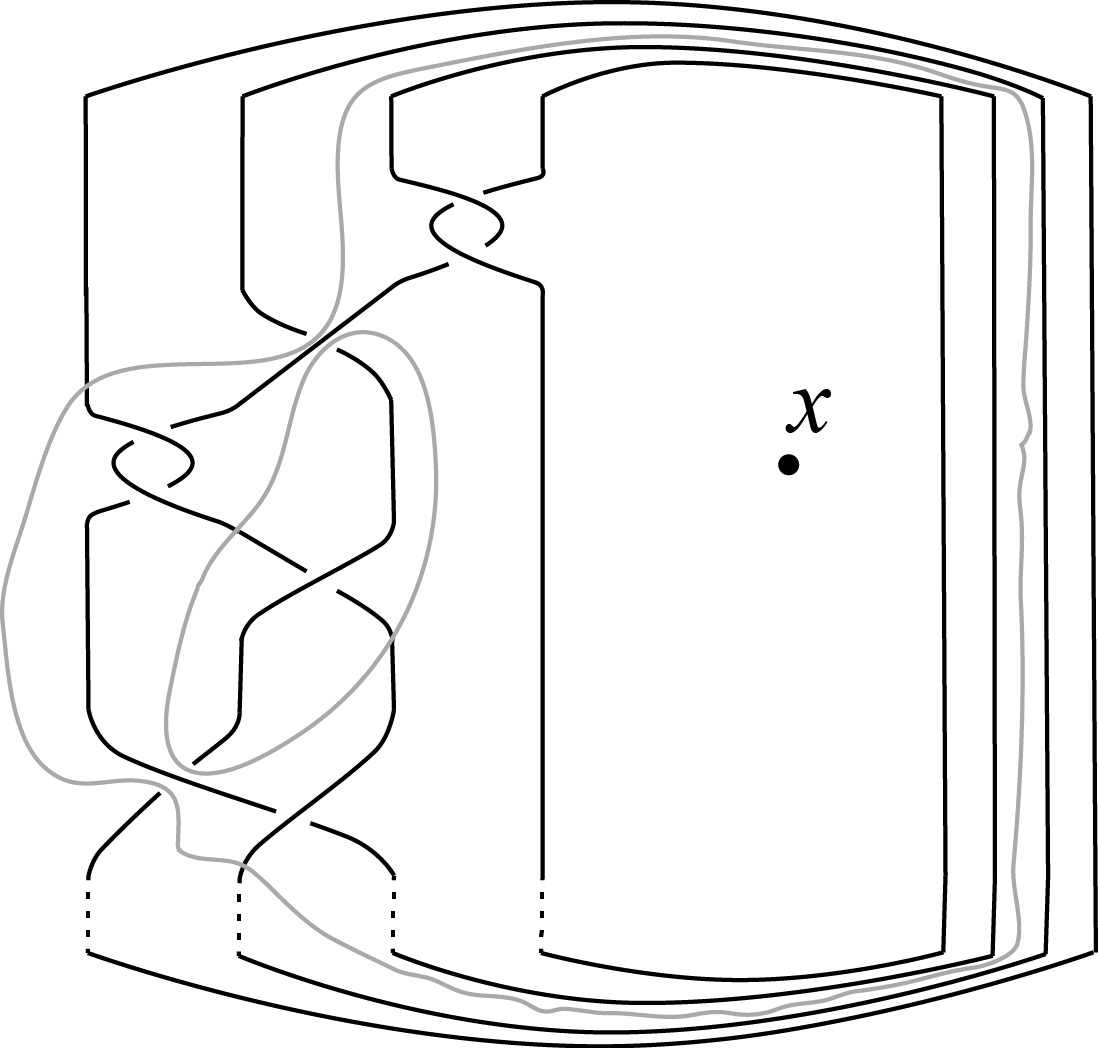}
 \caption{}
 \label{BraidExample}
 \end{subfigure}
 \caption{Diagrams of alternating braids}
 \end{figure}

 As before, let $D$ lie on the projection plane $Q$. A curve $C$ on $Q$ will be called \textit{special} with respect to $D$ if the following holds.
 \begin{enumerate}
\item $C$ intersects $D$ transversally in exactly four points. Denote them $x_1, x_2, x_3, x_4$.

\item $C$ intersects every edge of $D$ at most once.

\item $C$ is \textit{monotone}, \textit{i.e.} it can be isotoped (where the intersection points of $C$ with the link may slide along a link strand until they reach a crossing, but not further) so that a ray from $x$ always intersects $C$ in a single point. Fig.\ref{Curve1} shows an example of a monotone curve, and Fig.\ref{Curve2} shows an example of a curve that is not monotone. If there is a choice whether put $x$ inside or outside of $C$, while $x$ stays in the same region of $D$, we always assume $x$ is outside of $C$. This is illustrated in Fig.\ref{Curve3}, where the curve is not monotone.
 \end{enumerate}

Note that the last condition above also implies that, up to isotopy, $C$ winds exactly once around $x$.

 \begin{figure}[h]
  \begin{subfigure}[b]{0.325 \textwidth}
  \includegraphics[scale=0.41]{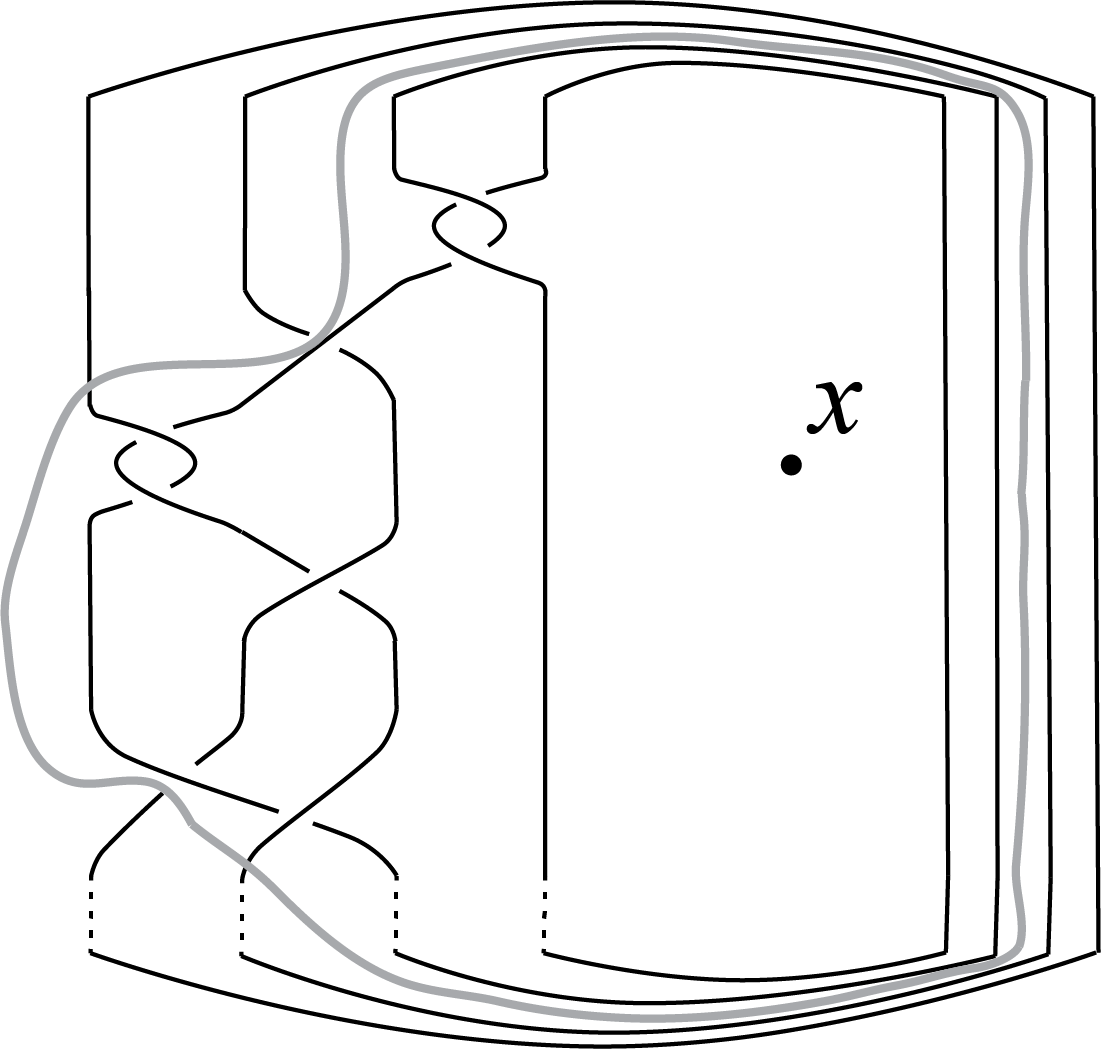}
  \caption{}
  \label{Curve1}
  \end{subfigure}
  \begin{subfigure}[b]{0.325 \textwidth}
  \centering
  \includegraphics[scale=0.41]{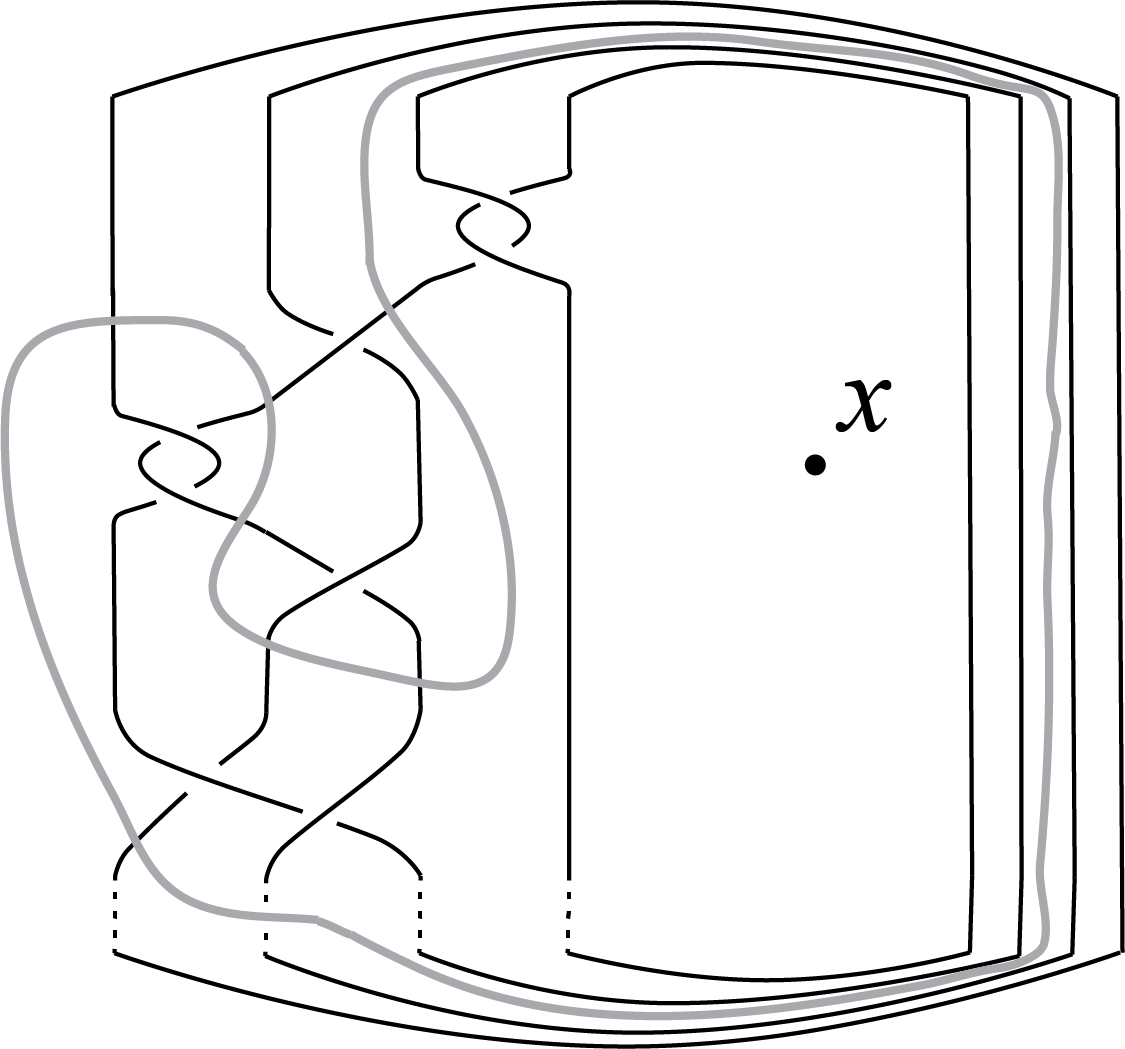}
  \caption{}
  \label{Curve2}
  \end{subfigure}
  \begin{subfigure}[b]{0.325 \textwidth}
  \centering
  \includegraphics[scale=0.41]{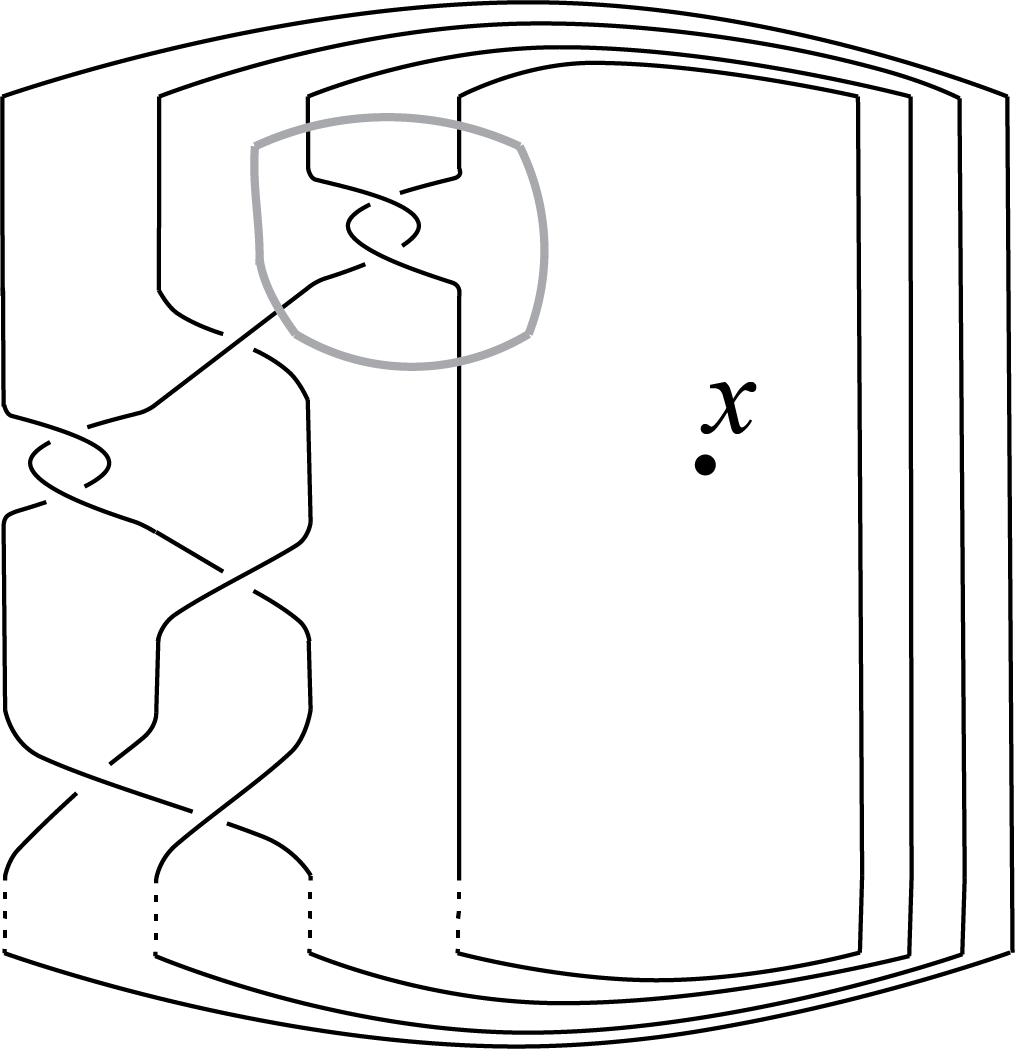}
  \caption{}
  \label{Curve3}
  \end{subfigure}
  \caption{In grey, a curve that is monotone (left) and two curves that are not (center, right).}
  \end{figure}

Recall that by Lemma \ref{Lemma} (\ref{GenusTwoWords}), a Conway sphere $F$ in standard position results in either one $PPPP$ or two $PSPS$ curves in $F\cap S^2_+$ (similarly, in $F\cap S^2_-$). Modify a $PSPS$ curve as follows. For every saddle it passes, push the curve from the saddle at a crossing into one of the edges adjacent to this crossing, as on Fig.\ref{PSPSmodified}. Call the new curve a \textit{modified $PSPS$ curve}. This modification does not correspond to an isotopy of $F$.

 \begin{theorem}\label{SpecialCurve}
 Suppose $D$ is a reduced alternating diagram of a prime alternating non-split $n$-braid $L$. Assume $S^3-L$ contains a Conway sphere $F$. Then either

 \begin{itemize}

 \item $D$ admits a special curve that is a $PPPP$ or modified $PSPS$ curve for $F$ or

 \item $D$ is  a diagram of the $3$-braid $\sigma_1\sigma_2^{-n_1}\sigma_1^{n_2}\sigma_2^{-1}\sigma_1^{n_3}\sigma_2^{-n_4}$, where $n_1, n_2, n_3, n_4$ are positive integers.\end{itemize}\end{theorem}
  \begin{proof}  Note that for $n=1$, we obtain an unknot that contains no essential surfaces. For $n=2$, we obtain a $(2, n)$-torus link, where an existence of a $PSPS$ curve contradicts Lemma \ref{Lemma} (\ref{SaddleArc}), and a $PPPP$ curve bounds a rational tangle. Hence, there are no Conway spheres, and the theorem holds vacuously. We therefore need to prove theorem for $n\geq 3$.

 We claim that for $F$, any $PPPP$ or modified $PSPS$ curve $C_1$ coming from $F$ either is special or bounds a 2-tangle whose diagram contains exactly one twist.

 The curve $C_1$ cuts $D$ into two 2-tangles, and travels around $x$ not more than once, since it does not have self-intersections. In addition, $C_1$ cannot intersect an edge of $D$ more than once by Lemma \ref{Lemma} (\ref{SaddleArc}, \ref{ArcTwice}).

 If a $PSPS$ curve enters a bigon through a saddle at a crossing (with $S$), it must exit the bigon through an edge (with $P$). But this contradicts Lemma \ref{Lemma} (\ref{SaddleArc}). Hence, if $C_1$ is a modified $PSPS$ curve, it cannot intersect a bigon. If $C_1$ is an actual $PPPP$ curve, it can be isotoped so that it does not cross any bigons.

 Therefore, we can depict $C_1$ on the diagram from Fig.\ref{BraidDiagram} so that if it passes through a black square, then the square represents a tangle with no crossings, and $C_1$ does not intersect $L$ in that square. The intersections of $C_1$ with $L$ can be grouped in consecutive pairs (a pair of intersections corresponds to $C_1$ entering/exiting a region of $D$), each pair arranged on a vertical, horizontal or diagonal line through a region of $D$. There are four intersections of $C_1$ with $D$, which can be denoted by $x_1, x_2, x_3, x_4$. $C_1$ is a closed curve, and once a region where it starts is chosen, the fourth intersection, $x_4$, must allow the curve to return to the same region.

 Fig.\ref{Intersections} shows all possible patterns once the starting region is chosen, and $D$ is placed on $S^2$. In the figure, $C_1$ is depicted in gray, and dotted lines mean that more twisting may occur there. Between two consecutive points of intersection, $x_i$ and $x_{i+1}$ (up to cyclic order), the curve may pass through some black squares that have no crossings in them, without intersecting the link. To see that these are all possible patterns, note that $C_1$ must bound at least one twist, but is allowed only four intersections with $L$. Call the twist $t$: it is represented by a black square in the figure. Hence, a part of $C_1$ is necessarily a vertical segment between intersections with two strands of $L$ coming out of $t$. Suppose the intersections are on the left of $t$, and denote them by $x_2, x_3$ (for another side, the argument is similar). Then $C_1$ must close up in some way on the right from $t$. Since there are just four intersections of $L$ with $C_1$, and $t$ already has four strands, $C_1$ cannot go proceed to the left before $x_2$ and after $x_3$. Fig.\ref{Intersections} then demonstrates all ways for $C_1$ to close up, up to a symmetry/reflection, where the fact that other segments of $C_1$ must be horizontal, vertical or diagonal, is used.

  \begin{figure}[h]
\centering
   \begin{subfigure}[b]{0.185 \textwidth}
 \centering
   \includegraphics[scale=0.66]{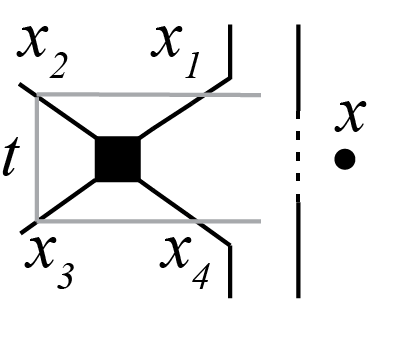}
   \caption{}
   \label{Intersections1}
   \end{subfigure}
   \begin{subfigure}[b]{0.185 \textwidth}
  \centering
     \includegraphics[scale=0.59]{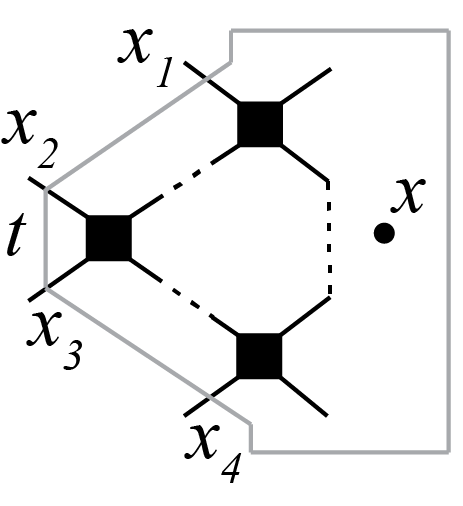}
     \caption{}
     \label{Intersections2}
     \end{subfigure}
     \begin{subfigure}[b]{0.185 \textwidth}
   \centering
       \includegraphics[scale=0.6]{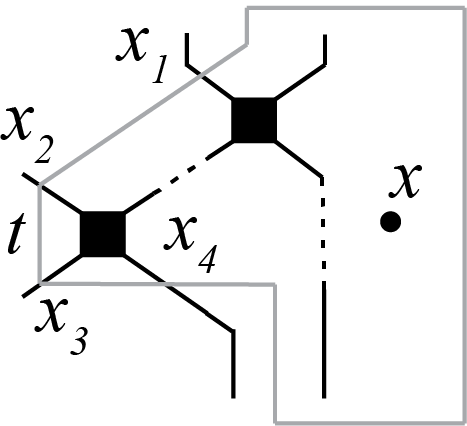}
       \caption{}
       \label{Intersections3}
       \end{subfigure}
   \begin{subfigure}[b]{0.185 \textwidth}
   \centering
   \includegraphics[scale=0.61]{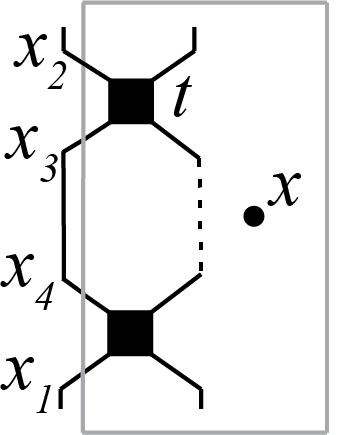}
   \caption{}
   \label{Intersections4}
   \end{subfigure}
   \begin{subfigure}[b]{0.185 \textwidth}
   \centering
   \includegraphics[scale=0.59]{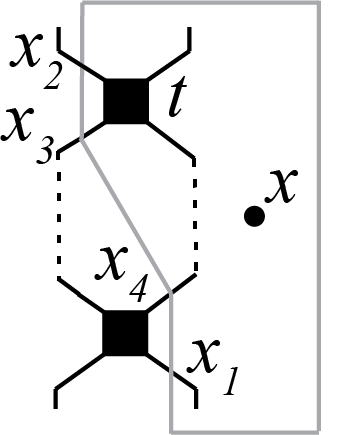}
   \caption{}
   \label{Intersections5}
   \end{subfigure}

   \caption{Patterns of intersection}
   \label{Intersections}
   \end{figure}

 Situation (1) includes either of the two possible ways of closing up $C_1$: either making a full circle around $x$, or through the shortest segment on the picture that connects its free ends. Fig.\ref{BraidExample} shows an actual example of a braid and two $PSPS$ curves that yield modified curves of types (1) and (3) from Fig.\ref{Intersections}.

 In each of the depicted situations, either $C_1$ bounds just one twist of $D$, or is monotone and therefore special. This concludes the proof of our claim.

 Now consider the $PPPP$ or $PSPS$ curve $C_1'$ represented by $C_1$. Assume $C_1$ is not special. Then by the claim it bounds just one twist.  If a $PPPP$ curve bounds just one twist, $F$ is compressible. Hence, only a $PSPS$ curve can bound a twist. The second $PSPS$ curve coming from $F$, denote it by $C_2'$, hits saddles at the same two crossings as $C_1'$. By the claim above, the modified $PSPS$ curve $C_2$ resulting from $C_2'$ either bounds a twist or is special itself.

 If $C_2$ is special, we are done. Otherwise each of $C_1$ and $C_2$ bounds a twist. Then there are six twists in the diagram: two enclosed by $C’_1$ and $C’_2$, two enclosed by the two modified $PSPS$ curves in $S^2_{-}$, and two one-crossing twists outside of all four of these modified $PSPS$ curves. This is depicted in Fig.\ref{6twists}, where $C_1, C_2$ are in grey color. Hence $L$ is a 3-braid. Each of the enclosed twists has just one crossing, since we isotoped $F$ so that a $PSPS$ curve does not go through a crossing of a bigon of $D$.  Therefore, $L$ is the braid $\sigma_1\sigma_2^{-n_1}\sigma_1^{n_2}\sigma_2^{-1}\sigma_1^{n_3}\sigma_2^{-n_4}$.  
 \end{proof}

  \begin{figure}[h]
\centering
   \includegraphics[scale=1.1]{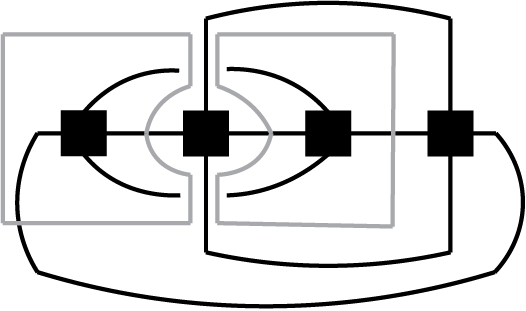}
   \caption{Two modified $PSPS$ curves are depicted in grey, and a link diagram is schematically depicted in black, with squares representing twists.}
   \label{6twists}
   \end{figure}

 \begin{example} Suppose that every black square on Fig.\ref{BraidDiagram} represents a tangle with at least one crossing. One can immediately see that such a diagram contains no special curves, and is not a diagram of a 3-braid. Therefore, by Theorem \ref{SpecialCurve}, $S^3-L$ contains no Conway sphere.
 \end{example}

\section{Acknowledgements}

We thank referee for the very careful reading of the paper. J. Hass acknowledges support from NSF grant DMS-1758107.   A. Thompson acknowledges support from NSF grant DMS-1664587. A. Tsvietkova acknowledges support from NSF DMS-1664425 (previously 1406588) and NSF DMS-2005496 grants, and by Insitute of Advanced Study (under DMS-1926686 grant).  All authors acknowledge support by Okinawa Institute of Science and Technology (OIST), Japan.

\newpage

Joel Hass \\
Department of Mathematics\\
University of California, Davis \\
One Shields Ave, Davis, CA 95616 \\
hass@math.ucdavis.edu

Abigail Thompson \\
Department of Mathematics\\
University of California, Davis \\
One Shields Ave, Davis, CA 95616 \\
thompson@math.ucdavis.edu

Anastasiia Tsvietkova\\
The Department of Mathematics and Computer Science\\
Rutgers University, Newark\\
101 Warren Street, Newark, NJ  07102\\
a.tsviet@rutgers.edu


\begin{thebibliography}{90}


\bibitem{BonSieb} F. Bonahon and L. Siebenmann, \textit{New Geometric Splittings of Classical Knots, and the Classification and Symmetries of Arborescent Knots}, Draft of a monograph.













\bibitem{HTT} J. Hass, A. Thompson, A. Tsvietkova, \textit{The number of surfaces of fixed genus in an alternating link diagram}, Int. Math. Res. Not. (2017), no. 6, 1611--1622.



\bibitem{KLic} R. C. Kirby, W. B. R. Lickorish, \textit{Prime knots and concordance}, Math. Proc. Cambridge
Philos. Soc. 86 (1979), 437--441.

\bibitem{LackenbyTunnel} M. Lackenby, \textit{Classification of alternating knots with tunnel number one}, Comm. Anal. Geom. 13 (2005), no. 1, 151--185.


\bibitem{Lickorish} W. B. R. Lickorish, \textit{Prime knots and tangles}, Trans. Amer. Math. Soc. 267 (1981), no. 1, 321--332.




\bibitem{Menasco1984} W. W.
Menasco, \textit{Closed incompressible surfaces in alternating
knot and link complements}, Topology {\bf 23} (1984), 37--44.

\bibitem{Menasco1985} W. Menasco, \textit{Determining incompressibility of surfaces in alternating knot and link complements}, Pacific J. Math. 117 (1985), no. 2, 353--370.



\bibitem{MT} W.W. Menasco, M. B. Thistlethwaite, \textit{The Tait flyping conjecture}, Bull. Amer. Math. Soc. (N.S.) 25 (1991), no. 2, 403--412.

\bibitem{MT93} W.W. Menasco, M. B. Thistlethwaite, \textit{The Classififcation of Alternating Links},  Ann. of Math. (2) 138 (1993), no. 1, 113--171.





\bibitem{Thistlethwaite} M. B. Thistlethwaite, \textit{On the algebraic part of an alternating link}, Pacific J. Math. 151 (1991), no. 2, 317--333.



\end{thebibliography}
\end{document}